\documentclass[10pt]{article}

\usepackage{graphics}
\usepackage{graphicx}

\usepackage[a4paper, left=35mm,right=35mm,top=34mm,bottom=34mm]{geometry}
\usepackage[utf8]{inputenc}
\usepackage[T1]{fontenc}
\usepackage[english]{babel}

\usepackage{enumerate}
\usepackage{graphicx}
\usepackage{hyperref}
\hypersetup{
    colorlinks=true,
    linkcolor=blue,
    filecolor=magenta,      
    urlcolor=cyan,
}

\usepackage{mathtools,amsthm,amssymb,amsfonts}
\usepackage{algorithm}
\usepackage{algorithmic}
\makeatother
\theoremstyle{plain}
\newtheorem{theorem}{Theorem}
\theoremstyle{definition}
\newtheorem{definition}[theorem]{Definition}
\theoremstyle{remark}

\usepackage{caption} 
\usepackage{fancyhdr}

\newcommand{\norm}[1]{\left\lVert#1\right\rVert}

\author{
  {\normalsize Qinmeng Zou}\thanks{CentraleSup\'elec, Universit\'e Paris-Saclay, France.}
  \and
  {\normalsize Fr\'ed\'eric Magoul\`es}\thanks{CentraleSup\'elec, Universit\'e Paris-Saclay, France
    (correspondence, frederic.magoules@hotmail.com).}
		}
\title{A New Cyclic Gradient Method Adapted to Large-Scale Linear Systems}
\date{}

\begin{document}
\maketitle
\thispagestyle{fancy}

\begin{abstract}
\noindent This paper proposes a new gradient method to solve the large-scale problems.
Theoretical analysis shows that the new method has finite termination property for two dimensions and converges R-linearly for any dimensions.
Experimental results illustrate first the issue of parallel implementation.
Then, the solution of a large-scale problem shows that the new method is better than the others, even competitive with the conjugate gradient method.
\end{abstract}


\section{Introduction}

We are interested in investigating new gradient methods for the solution of linear system
\begin{equation}
\label{eq:ls}
Ax=b,
\end{equation}
where $A\in\mathbb{R}^{n\times n}$ is symmetric positive definite (SPD) and $b\in\mathbb{R}^n$.
This problem is equivalent to the minimization of a convex quadratic function
\begin{equation}
\label{eq:cq}
f(x) = \frac{1}{2}x^\intercal Ax - b^\intercal x.
\end{equation}
Gradient methods generate a sequence of the form
\begin{equation}
\label{eq:grd}
x_{k+1} = x_k - \alpha_k g_k,\quad k = 0, 1, \dots,
\end{equation}
where $g_k = Ax_k - b$.
It is well known that the steepest descent (SD) method \cite{Cauchy1847} performs poor in most cases, where the steplength can be written as follows
\begin{equation}
\label{eq:sd}
\alpha_k^\text{SD} = \frac{g_k^\intercal g_k}{g_k^\intercal Ag_k}.
\end{equation}
The iterates generated tend to asymptotically alternate between two directions \cite{Akaike1959}.
In contrast, the conjugate gradient (CG) method \cite{Hestenes1952} is often the method of choice that will terminate in at most $n$ iterations.
It is very attractive because of its high efficiency and low storage requirement.
Nonetheless, CG iteration depends strongly on the search of direction calculation, i.e., any derivation such as round-off errors can seriously degrade performance \cite{Fletcher2005}.

In the past several decades, a renewed interest for gradient methods has appeared since Barzilai and Borwein \cite{Barzilai1988} proposed two efficient nonmonotone steplengths
\begin{equation}
\label{eq:bb1}
\alpha_k^\text{BB1} = \frac{g_{k-1}^\intercal g_{k-1}}{g_{k-1}^\intercal A g_{k-1}}.
\end{equation}
\begin{equation}
\label{eq:bb2}
\alpha_k^\text{BB2} = \frac{g_{k-1}^\intercal Ag_{k-1}}{g_{k-1}^\intercal A^2 g_{k-1}}.
\end{equation}
The motivation consists in approximating the Hessian and imposing some quasi-Newton properties.
Some theories and experiments have shown that BB methods have good performance and are competitive with CG methods when low accuracy is required or small perturbation exists \cite{Fletcher2005}.
The convergence has been proven by Raydan \cite{Raydan1993}.
Furthermore, Friedlander et al. \cite{Friedlander1999} provided a general framework under the name of ``gradient method with retards'' that SD and BB both belong to it, as well as several alternate methods proposed later \cite{Dai2003a, Dai2003b, Zhou2006}.

Motivated by the two-dimensional finite termination property, Yuan \cite{Yuan2006} provided a somewhat complicated steplength
\begin{equation}
\label{eq:yuan}
\alpha_k^\text{Y} = \frac{2}{\sqrt{\left(\frac{1}{\alpha_{k-1}^\text{SD}} - \frac{1}{\alpha_k^\text{SD}}\right)^2 + \frac{4g_k^\intercal g_k}{s_{k-1}^\intercal s_{k-1}}} + \frac{1}{\alpha_{k-1}^\text{SD}} + \frac{1}{\alpha_k^\text{SD}}},
\end{equation}
where $s_{k-1} = x_k - x_{k-1}$.
He gave two algorithms and some variants were investigated further by Dai and Yuan \cite{Dai2005c}.
Among these methods, the second variant (DY) is the most efficient one according to the experiments in \cite{Dai2005c}, where iterates are generated of the form
\begin{equation}
\label{eq:dy}
\alpha_k^\text{DY} =
\begin{cases}
\alpha_k^\text{SD}, & k\bmod 4 < 2, \\
\alpha_k^\text{Y}, & \text{otherwise}.
\end{cases}
\end{equation}

In this paper, we address the properties of cyclic gradient methods, especially their parallel behavior.
We propose a new algorithm based on the Yuan steplength, which has also the two-dimensional finite termination property.
In the next section, we introduce the cyclic gradient methods and propose our new steplength.
In Section~\ref{sec:3}, we give the convergence results of the new method.
Some numerical results are presented in Section~\ref{sec:4}.
Finally, a concluding remark is shown in Section~\ref{sec:5}.

\section{Cyclic Gradient Methods}
\label{sec:2}

Friedlander et al. \cite{Friedlander1999} proposed an ingenious framework that gives rise to a great number of potentially efficient algorithms.
Firstly, assume that $m\in\mathbb{N}$ represents retard that allows to employ the information from previous iterations.
Let
\begin{equation}
\bar{k} = \max\{0,\ k-m\},
\end{equation}
then a collection of possible choices of steplength can be set as follows
\begin{equation}
\label{eq:gmr}
\alpha_k^\text{GMR} = \frac{g_{\tau(k)}^\intercal A^{\rho(k)}g_{\tau(k)}}{g_{\tau(k)}^\intercal A^{\rho(k)+1}g_{\tau(k)}},
\end{equation}
where
\begin{equation}
\tau(k)\in\left\{\bar{k},\ \bar{k}+1,\ \dots,\ k-1,\ k\right\},
\end{equation}
and
\begin{equation}
\rho(k)\in\left\{q_1,\ \dots,\ q_m\right\},\quad q_j \ge 0,
\end{equation}
where $k\in\mathbb{N}$.
The next theorem summarizes the convergence result in \cite{Friedlander1999}.
\begin{theorem}[Friedlander et al., 1999]
\label{thm:Friedlander1999}
Consider the linear system \eqref{eq:ls} with $A\in\mathbb{R}^{n\times n}$ is SPD and $b\in\mathbb{R}^n$, where $x_*=A^{-1}b$ is the exact solution.
Consider the gradient method \eqref{eq:grd} being used to solve \eqref{eq:ls} and the steplength $\alpha_k$ given by \eqref{eq:gmr}.
Then the sequence $\{x_k\}$ converges to $x_*$ starting from any point $x_0$.
\end{theorem}
For a proof of the above theorem, see \cite{Friedlander1999}.
Incidentally, several potential algorithms were provided therein, including the first cyclic gradient method under the name of cyclic steepest descent (CSD) as suggested in \cite{Dai2003a}, which can be summarized as follows
\begin{equation}
\label{eq:csd}
\alpha_k^\text{CSD} =
\begin{cases}
\alpha_k^\text{SD}, & k\bmod m = 0, \\
\alpha_{k-1}, & \text{otherwise}.
\end{cases}
\end{equation}
Notice that if we choose $\rho(k) = 0$ and $\tau(k) = \bar{k}+1,\ \dots,\ k-1,\ k$, then \eqref{eq:gmr} becomes CSD method, which satisfies the Theorem~\ref{thm:Friedlander1999}.
On the other hand, Dai \cite{Dai2003a} proposed a variant called cyclic Barzilai-Borwein (CBB) method.
They suggested that
\begin{equation}
\label{eq:cbb}
\alpha_k^\text{CBB} =
\begin{cases}
\alpha_k^\text{BB1}, & k\bmod m = 0, \\
\alpha_{k-1}, & \text{otherwise}.
\end{cases}
\end{equation}
Similarly, if we choose $\rho(k) = 0$ and $\tau(k) = \bar{k},\ \bar{k}+1,\dots,\ k-1$, then \eqref{eq:gmr} becomes CBB method.

Although these methods greatly speed up the convergence, their motivation is too straightforward to further accelerate the iterations, which relies on the nonmonotone property to search the whole space without sink into any lower subspace spanned by eigenvectors \cite{Fletcher2005}.
This allows to reduce the gradient components more or less in the same asymptotic rate \cite{Dai2005c}.

The recent literature showed that Yuan steplength may lead to efficient algorithms \cite{Yuan2006, Dai2005c}.
All methods therein have two-dimensional finite termination property, i.e., if \eqref{eq:dy} is applied to a linear system in two-dimensional space, then the algorithm will terminate in at most 3 iterations.
In general, such property seems not attractive in practice.
However, experiments showed that they perform well in higher dimensions and are competitive with BB methods for large-scale problems \cite{Dai2005c}.

Inspired by the Yuan steplength, we suggest a simple way of modifying steepest descent model to a cyclic gradient method.
Consider a steplength of the form
\begin{equation}
\label{eq:yb}
\alpha_k^\text{YB} =
\begin{cases}
\alpha_k^\text{SD}, & k\bmod 3 = 0 \text{ or } 2, \\
\alpha_k^\text{Y}, & k\bmod 3 = 1.
\end{cases}
\end{equation}
Here we modify the order of SD and Y compared to the original YB formula, which is useful for the development of the new algorithm.
Apart from this change, \eqref{eq:yb} is indeed the second algorithm propose by the pioneering work of Yuan \cite{Yuan2006}.
It keeps the two-dimensional finite termination property that performs as well as BB for large-scale problems and better for small-scale problems.
We could introduce simply the cyclic behavior based on \eqref{eq:yb} of the form
\begin{equation}
\forall m\in\mathbb{N},\text{ if }k\bmod (3+m) > 2,\text{ then }\alpha_k = \alpha_{k-1}.
\end{equation}
Besides, we find that De Asmundis et al. \cite{DeAsmundis2013} gives an interesting view about the iterations of SD method, where the technique of alignment was proposed therein to force the gradients into one-dimensional subspace and avoid the zigzag pattern.
Notice that the inverse of constant Rayleigh quotient such as SD and BB steplengths has the similar property.
Thus, constant SD with retards can also give rise to the alignment behavior and keep the nonmonotone benefit.
To achieve this goal, we need to impose a repeat time to the zigzag process.
Meanwhile, we want to keep the process based on Yuan steplength in the first several iterations.
These motivations lead to a new method of the form
\begin{equation}
\label{eq:cy}
\alpha_k^\text{CY} =
\begin{cases}
\alpha_k^\text{Y}, & k\bmod (l+m+2) = 1, \\
\alpha_k^\text{SD}, & k\bmod (l+m+2) < l+2, \\
\alpha_{k-1}, & \text{otherwise},
\end{cases}
\end{equation}
where $l\ge 1$ and $m\ge 1$.
Such formula seems complicated, but indeed easy to understand.
There are three components consisting in \eqref{eq:cy}: the first SD and Y are used to insure the finite termination property; the parameter $l$ acting on the second part of SD is used to keep several zigzag iterations; finally, the retard term $m$ induces alignment and provides nonmonotone behavior to leap from the lower subspace.

\section{Convergence Analysis}
\label{sec:3}

By the invariance property under any orthogonal transformation, we can assume without loss of generality that
\begin{equation}
\label{eq:A}
A = \text{diag}(\lambda_1,\ \dots,\ \lambda_n),
\end{equation}
where
\begin{equation}
\label{eq:lambda}
1=\lambda_1\le\dots\le\lambda_n.
\end{equation}
We follow the convergence framework established by Dai \cite{Dai2003a} and adapt it to our method.
Let
\begin{equation}
\label{eq:pa}
G(k,\mu) = \sum_{i=1}^\mu g_{i,k}^2,
\end{equation}
where $g_{i,k}$ is the $i$th component of $g_k$.
A preliminary property is defined as follows.
\begin{definition}[Property A]
Suppose that matrix $A$ has the form \eqref{eq:A} with condition \eqref{eq:lambda} holds.
If $\exists \xi\in\mathbb{N}$, $\exists M_1,M_2>0$, such that
$\forall\mu\in\{1,\dots,n-1\}$, $\forall\epsilon>0$, $\forall j\in\{0,\dots,\min\{k,\xi\}-1\}$,
\begin{itemize}
\item $\lambda_1\le\alpha_k^{-1}\le M_1$;
\item if $G(k-j,\mu)\le\epsilon$ and $g_{\mu+1,k-j}^2\ge M_2\epsilon$, then $\alpha_k^{-1}\ge\frac{2}{3}\lambda_{\mu+1}$,
\end{itemize}
then the steplength $\alpha_k$ has Property A.
\end{definition}
\noindent The convergence framework of Dai can be deduced from Property A, stated as follows.
\begin{theorem}[Dai, 2003]
\label{thm:Dai2003}
Consider the linear system \eqref{eq:ls} with $A\in\mathbb{R}^{n\times n}$ of the form \eqref{eq:A} and $b\in\mathbb{R}^n$.
Consider the gradient method \eqref{eq:grd} being used to solve \eqref{eq:ls}.
If the steplength $\alpha_k$ has Property A, then the sequence $\{\norm{g_k}\}$ converges to 0 R-linearly for any starting point $x_0$.
\end{theorem}
For a proof of the above theorem, see \cite{Dai2003a}.
Many gradient methods have Property A as mentioned in \cite{Dai2003a}, e.g., the gradient method with retards \eqref{eq:gmr}.
Inspired by the demonstration therein, we now develop a convergence result for the CY method.
\begin{theorem}
\label{thm:cv2}
Consider the linear system \eqref{eq:ls} with $A\in\mathbb{R}^{n\times n}$ of the form \eqref{eq:A} and $b\in\mathbb{R}^n$.
Consider the gradient method \eqref{eq:grd} with steplength \eqref{eq:cy} being used to solve \eqref{eq:ls}.
Then the steplength $\alpha_k^\text{CY}$ has Property A.
\end{theorem}
\begin{proof}
Note that \eqref{eq:cy} has three alternate steplengths, whereas the SD updating process and the constant process using the last SD steplength both follow the framework \eqref{eq:gmr}, which has been proven to have the Property A \cite{Dai2003a}.
Therefore, we only investigate the Yuan steplength.

Recall that Yuan steplength has the following property
\begin{equation}
\left(\frac{1}{\alpha_{k-1}^\text{SD}} + \frac{1}{\alpha_k^\text{SD}}\right)^{-1} < \alpha_k^\text{Y} < \min\left\{\alpha_{k-1}^\text{SD},\ \alpha_k^\text{SD}\right\},
\end{equation}
which is given in \cite{Yuan2006}.
Hence,
\begin{equation}
\lambda_1 \le \frac{1}{\alpha_k^\text{SD}} < \frac{1}{\alpha_k^\text{Y}} < \frac{1}{\alpha_{k-1}^\text{SD}} + \frac{1}{\alpha_k^\text{SD}} \le 2\lambda_n.
\end{equation}
Then the first condition of Property A holds by setting $M_1 = 2\lambda_n$.
For the second one, let $M_2 = 2$ and $\xi = 1$, which yields $j = 0$.
Suppose that
\begin{equation}
G(k,\mu)\le\epsilon,\quad g_{\mu+1,k}^2\ge M_2\epsilon,
\end{equation}
for all $\mu\in\{1,\dots,n-1\}$, and $\epsilon>0$.
Hence, the inverse of Yuan steplength becomes
\begin{equation}
\begin{split}
\frac{1}{\alpha_k^\text{Y}} & > \frac{1}{\alpha_k^\text{SD}} = \frac{g_k^\intercal Ag_k}{g_k^\intercal g_k} = \frac{\sum_{i=1}^n \lambda_i g_{i,k}^2}{\sum_{i=1}^n g_{i,k}^2} \\
& \ge \frac{\lambda_{\mu+1}\sum_{i=\mu+1}^n g_{i,k}^2}{\sum_{i=1}^\mu g_{i,k}^2 + \sum_{i=\mu+1}^n g_{i,k}^2} \\
& \ge \frac{\lambda_{\mu+1}}{\frac{\sum_{i=1}^\mu g_{i,k}^2}{g_{\mu+1,k}^2} + 1} \\
& \ge \frac{\lambda_{\mu+1}}{\frac{\epsilon}{2\epsilon}+1} = \frac{2}{3}\lambda_{\mu+1}
\end{split}
\end{equation}
Hence, the second condition of Property A is satisfied, which completes the proof.
\end{proof}

\section{Numerical Results}
\label{sec:4}

We first address the issue of parallel implementation.
The dot product is engaged in the computation of steplength, which is the major obstacle of parallelization.
Here we have two strategies to realize this goal.
Let $A_i$ be the band matrix stored in the $i$th processor.
The first one (Gather Algorithm, GA) is to gather the vector $q_i = A_i * g$ and execute dot product with global vectors, shown as follows
\begin{algorithmic}
\STATE $\texttt{Allgatherv}(q,\ q_i)$
\STATE $\alpha = \texttt{Dot}(g,\ g)\ /\ \texttt{Dot}(g,\ q)$
\end{algorithmic}
Then, the second one (Reduce Algorithm, RA) consists in computing the dot product locally, shown as follows
\begin{algorithmic}
\STATE $c_i = \texttt{Dot}(g_i,\ q_i)$
\STATE $\texttt{Allreduce}(c,\ c_i,\ \texttt{SUM})$
\STATE $\alpha = \texttt{Dot}(g,\ g)\ /\ c$
\end{algorithmic}
Besides, we can see that global gradient vector is used in each iteration that we must proceed another \texttt{Allgatherv} function to communicate with other processor.
\begin{figure}[!ht]
\centering
\includegraphics[width=3.6in]{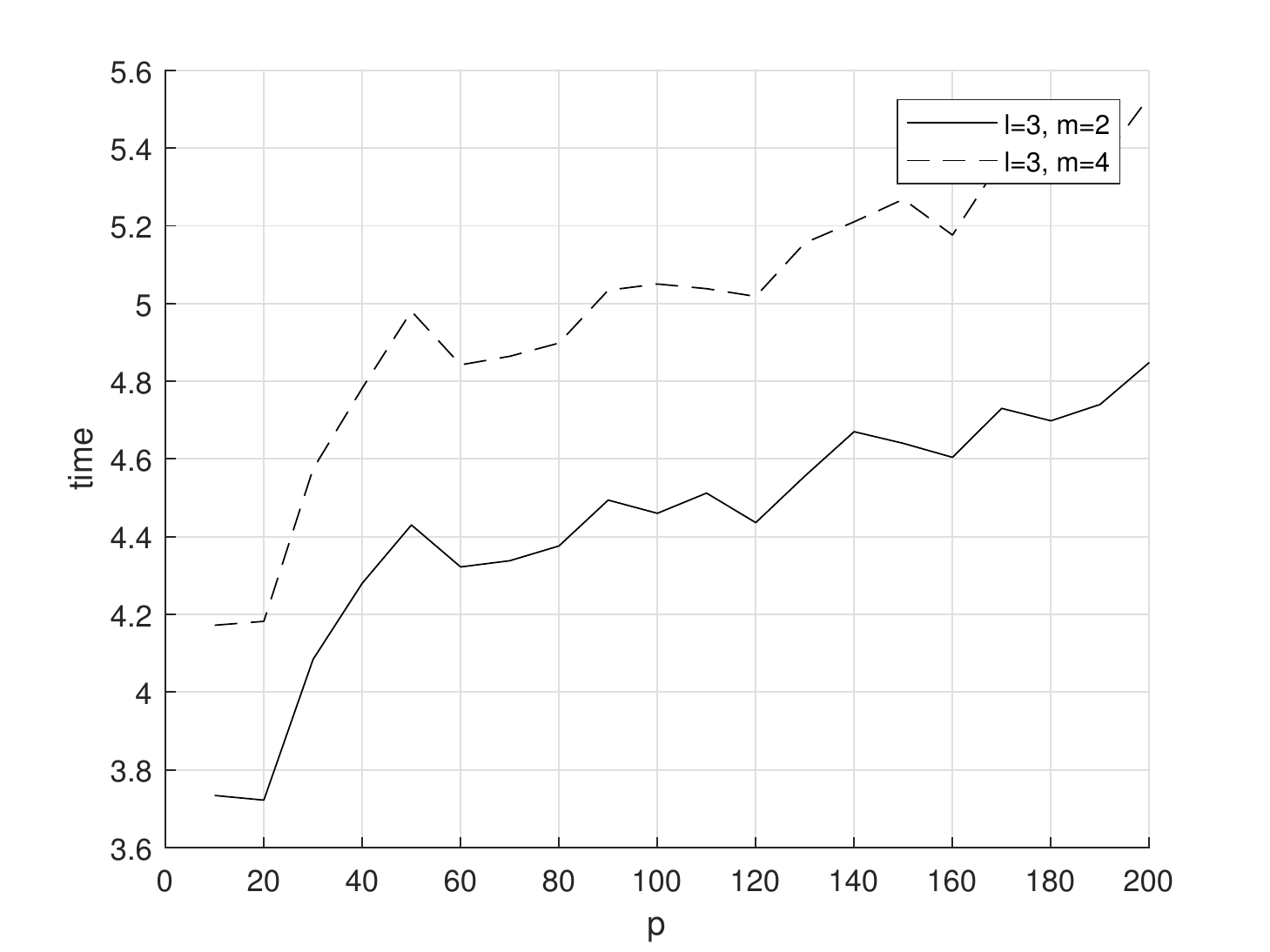}
\caption{Parallel CY method with GA implementation}
\label{fig:1}
\end{figure}
\begin{figure}[!ht]
\centering
\includegraphics[width=3.6in]{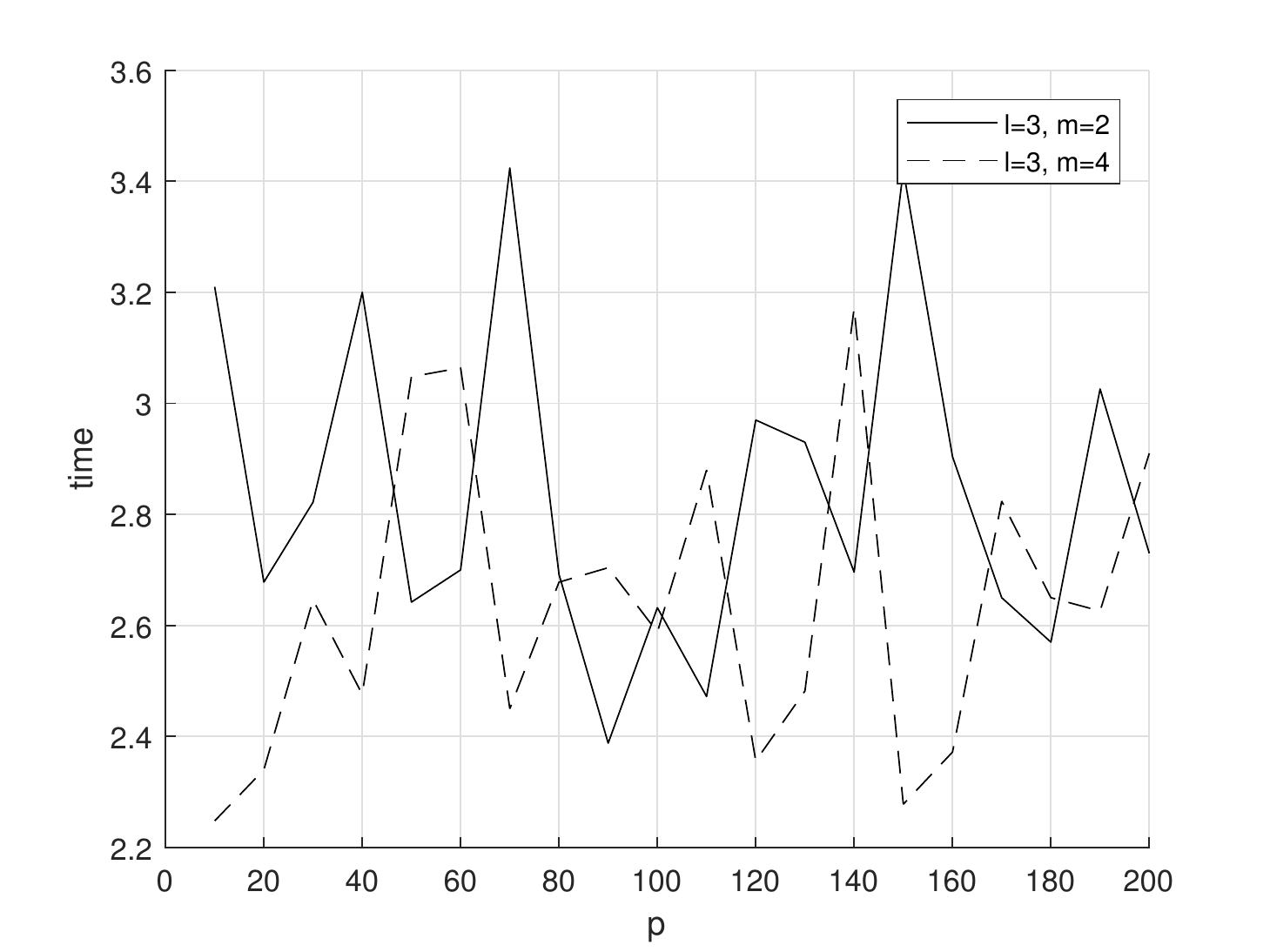}
\caption{Parallel CY method with RA implementation}
\label{fig:2}
\end{figure}
Let $p$ be the number of processors.
The two experiments are proceeded by Alinea \cite{Magoules2015a} (see also, e.g., \cite{Magoules2017a, Magoules2018c, Magoules2018e}) and JACK \cite{Magoules2017b, Magoules2018b} (see also, e.g., \cite{Magoules2018a}) and results are illustrated in Figures~\ref{fig:1} and~\ref{fig:2}.
We can see that generally the results are not good because the first one imposes so much computation and communication load, while the second one causes indeed the problem of loss of precision.
These problems exist in all projection methods and by now we have not yet managed to find a solution.

The second experiments are proceeded by Matlab R2017b with a large-scale problem provided by The SuiteSparse Matrix Collection \cite{Davis2011}, with $n=50000$ and $349968$ non-zero values.
All parameters are chosen under a training problem given in \cite{Barzilai1988}, such that $l=4,m=3$ for CY, $m=3$ for CSD, and $m=4$ for CBB.
Average results are shown in Table~\ref{tab:1}.
\begin{table}[!t]
\renewcommand{\arraystretch}{1.3}
\caption{Gradient methods with different residual threshold}
\label{tab:1}
\centering
\begin{tabular}{|c||c|c|c|c|c|c|}
\hline
& $10^{-1}$ & $10^{-2}$ & $10^{-3}$ & $10^{-4}$ & $10^{-5}$ & $10^{-6}$ \\
\hline
CG & $58$ & $735$ & $2617$ & $\textbf{4251}$ & $\textbf{5786}$ & $\textbf{7535}$ \\
\hline
CY & $\textbf{13}$ & $\textbf{208}$ & $\textbf{1153}$ & $\textbf{4275}$ & $\textbf{6181}$ & \textbackslash \\
\hline
CSD & $27$ & $212$ & $\textbf{1470}$ & $4357$ & $6713$ & \textbackslash \\
\hline
CBB & $31$ & $241$ & $1965$ & $7534$ & \textbackslash & \textbackslash \\
\hline
DY & $\textbf{15}$ & $\textbf{200}$ & $1595$ & $6415$ & \textbackslash & \textbackslash \\
\hline
BB1 & $356$ & $984$ & $2494$ & $4612$ & $8633$ & \textbackslash \\
\hline
SD & $61$ & $5773$ & \textbackslash & \textbackslash & \textbackslash & \textbackslash\\
\hline
\end{tabular}
\end{table}
We use bold numbers indicating the most efficient algorithms under each residual threshold.
Backslash represents a number of iterations more than $10000$.
From Table~\ref{tab:1}, we see that the CY method performs better than other methods except CG.
Although we do not yet beat CG in high precision, our method is still competitive in most cases.
Specifically, we can see that CY is stable throughout the iterations and much better than CG when low accuracy is required.

\section{Conclusions}
\label{sec:5}

In this paper, we have proposed a new gradient method and shown that it is very competitive with CG method and better than the others for large-scale problems.
However, it is still lack of theoretical evidence supporting such results.
It is also important to find a better parallelization strategy.
Therefore it still remains to study the properties and high performance implementation of gradient methods.

\section*{Acknowledgment}
This work was supported by the French national programme LEFE/INSU and the project ADOM (M\'ethodes de d\'ecomposition de domaine asynchrones) of the French National Research Agency (ANR).

\bibliography{ref}
\bibliographystyle{abbrv}

\end{document}